\title{Equidistant Codes in the Grassmannian}
\author{Tuvi Etzion \and Netanel Raviv}
\date{\today}

\documentclass[12pt]{article}

\usepackage{filecontents}

\usepackage[margin=0.7in]{geometry}

\usepackage{amssymb}
\usepackage{amsmath}
\usepackage{amsthm}
\usepackage{units}

\newtheorem{theorem}{Theorem}
\newtheorem{definition}{Definition}

\newtheorem{lemma}{Lemma}
\newtheorem{corollary}{Corollary}
\newtheorem{conjecture}{Conjecture}

\numberwithin{subcase}{case}

\newtheorem{remark}{Remark}

\newcommand{\cl}[1]{\mathcal{#1}}

\newcommand{\Fq}{\mathbb{F}_{q}}

\newcommand{\ga}{\alpha}

\newcommand{\sus}{\subseteq}

\newcommand{\bC}{\bf C}
\newcommand{\C}{\mathbb{C}}

\newcommand{\T}{\mathbb{T}}

\newcommand{\bF}{\mathbb{F}}
\newcommand{\F}{\mathbb{F}}

\newcommand{\bD}{\mathbb{D}}

\newcommand{\bS}{\mathbb{S}}

\newcommand{\bP}{\mathbb{P}}
\newcommand{\bM}{\mathbb{M}}
\newcommand{\cE}{\mathcal{E}}

\newcommand{\cS}{\mathcal{S}}
\newcommand{\cA}{\mathcal{A}}
\newcommand{\cC}{\mathcal{C}}

\newcommand{\cP}{\mathcal{P}}
\newcommand{\cB}{\mathcal{B}}
\newcommand{\qbin}[3]{{#1 \brack #2}_{#3}}
\newcommand{\grsmn}[3]{\cl{G}_{#1}\left(#2,#3\right)}
\newcommand{\Span}[1]{{\left\langle {#1} \right\rangle}}

\makeatletter
\renewcommand*\env@matrix[1][*\c@MaxMatrixCols c]{%
  \hskip -\arraycolsep
  \let\@ifnextchar\new@ifnextchar
  \array{#1}}
\makeatother

\newcommand{\cG}{{\cal G}}

\makeatletter
 \DeclareRobustCommand{\nsbinom}{\genfrac[]\z@{}}
 \makeatother
 
 \newcommand{\sbinomq}[2]{\nsbinom{{#1}}{{#2}}_{q}}
  \newcommand{\sbinomtwo}[2]{\nsbinom{{#1}}{{#2}}_{2}}

\begin{document}
\maketitle

\begin{abstract}
Equidistant codes over vector spaces are considered.
For $k$-dimensional subspaces over a large vector space
the largest code is always a sunflower. We present several
simple constructions for such codes which might produce
the largest non-sunflower codes. A novel construction,
based on the Pl\"{u}cker embedding, for 1-intersecting codes
of $k$-dimensional subspaces over $\F_q^n$, $n \geq \binom{k+1}{2}$,
where the code size is $\frac{q^{k+1}-1}{q-1}$ is presented.
Finally, we present a related construction which generates equidistant
constant rank codes with matrices of size $n \times \binom{n}{2}$
over $\F_q$, rank $n-1$, and rank distance $n-1$.
\end{abstract}


\footnotetext[1]{This research was supported in part by the Israeli
Science Foundation (ISF), Jerusalem, Israel, under
Grant 10/12.

The work of Netanel Raviv is part of his Ph.D.
thesis performed at the Technion.

The authors are with the Department of Computer Science, Technion,
Haifa 32000, Israel.

e-mail: etzion@cs.technion.ac.il, netanel.raviv@gmail.com .}

\section{Introduction}

Equidistant codes in the Hamming scheme are the bridge between coding
theory and extremal combinatorics. A code is called equidistant if the
distance between any two distinct codeword is equal to a given parameter $d$.
A code which is represented as collection of subsets
will be called $t$-\emph{intersecting}
if the intersection between any two codewords is exactly $t$.
A binary code of length $n$ in the Hamming scheme is a collection of
binary words of length $n$. A binary code of length $n$ with $m$ codewords
can be represented as a binary $m \times n$ matrix whose rows are
the codewords of the code. The Hamming distance between two codewords
is the number of positions in which they differ. The \emph{weight}
of a word in the number of nonzero entries in the word. A code $\bC$ is called
\emph{a constant weight code} if all its codewords have the same weight.
The \emph{minimum Hamming distance} of the code is the smallest
distance between two distinct codewords. An $(n,d,w)$ code is
a constant weight code of length $n$, weight $w$,
and minimum Hamming distance $d$.
Optimal $t$-intersecting equidistant codes (which must be also constant weight)
have two very interesting families of codes, codes obtained from projective
planes and codes obtained from Hadamard matrices~\cite{Dez73}.

Recently, there have been lot of new interest in codes whose codewords
are vector subspaces of a given vector space over $\F_q$, where $\F_q$
is the finite field with $q$ elements. The interest in these codes
is a consequence of their application in random network coding~\cite{KoKs08}.
These codes are related to what
are known as $q$-analogs. The well known concept of $q$-analogs replaces subsets by
subspaces of a vector space over a finite field and their sizes by
the dimensions of the related subspaces. A binary codeword can
be represented by a subset whose elements are the
nonzero positions. In this respect, \emph{constant dimension
codes} are the $q$-analog of constant weight codes.
For a positive integer $n$, the set of all subspaces of~$\F_q^n$
is called the \emph{projective space}~$\cP_q(n)$.
The set of all $k$-dimensional subspaces ($k$-subspaces in short)
of~$\F_q^n$ is called a \emph{Grassmannian} and is denoted by~$\cG_q(n,k)$.
The size of $\cG_q(n,k)$ is given by the $q$-binomial
coefficient $\sbinomq{n}{k}$, i.e.
$$
| \cG_q (n,k) | = \sbinomq{n}{k} = \frac{(q^n-1)(q^{n-1}-1) \cdots (q^{n-k+1}-1)}{(q^k-1)(q^{k-1}-1) \cdots (q-1)}~.
$$
Similarly, the
set of all $k$-subspaces of a subspace $V$ is denoted by $\qbin{V}{k}{q}$.
It turns out that the natural measure of distance in $\cP_q(n)$ is given by
$$
d_S (X,Y)  = \dim X+ \dim Y -2 \dim\bigl( X\, {\cap} Y\bigr)~,
$$
for all $X,Y \in \cP_q(n)$. This measure of distance is called the
\emph{subspace distance} and it is the $q$-analog of the
Hamming distance in the Hamming space. Finally,
an $[n,d,k]_q$ code is a subset of $\cG_q(n,k)$ (constant dimension code), having
minimum subspace distance $d$, where the minimum subspace
distance of the code is the smallest subspace distance
between any two codewords. In the sequel, when distance
will be mentioned it will be understood from the context if the
Hamming distance or the subspace distance is used.

The main goal of this paper is to consider the
$t$-intersecting constant dimension codes (which are clearly
equidistant). In this context interesting constructions
of such codes were given
in~\cite{Han12,han13}, but their size is rather small.

Some optimal binary equidistant codes form
a known structure from extremal combinatorics
called a partial projective plane.
This structure was defined and studied in~\cite{Hall77-PartialPG}.
An important concept in this context is the sunflower.
A binary constant weight code of weight $w$ is
called a \emph{sunflower} if any two codewords
intersect in the same $t$ coordinates.
A \emph{sunflower} $\bS \subset \cG_q(n,k)$ is a $t$-intersecting
equidistant code in which any two codewords $X, Y \in \bS$ intersect in
the same $t$-subspace $Z$. The $t$-subspace $Z$ is called the \emph{center}
of $\bS$ and is denoted by $Cen(\bS)$.

An upper bound on the size of $t$-intersecting binary constant weight code
with weight $w$ was given in~\cite{Dez73,DeFr81}. If the size of
such code is greater than $(w-t)^2 + (w-t) +1$ then the code
is a sunflower. This bound is attained when $t=1$, $w=q+1$,
where $q$ is a prime power and the codewords are the characteristic
vectors of length $\sbinomq{n}{1}$ which represent
the lines of the projective plane of order~$q$. Except for some
specific cases~\cite{HJKvL77-Distance12} no better bound is known. This bound
can be adapted for $t$-intersecting constant dimension codes
with dimension $k$ over $\F_q$. We can view the characteristic
binary vectors with weight $\frac{q^k-1}{q-1}$ as the codewords,
which implies intersection of size $\frac{q^t-1}{q-1}$ between codewords
and hence we have the following bound.

\begin{theorem}
\label{thm:sunflower_bound}
If a $t$-intersecting constant dimension code of dimension $k$
has more than $\left( \frac{q^k-q^t}{q-1} \right)^2 + \frac{q^k-q^t}{q-1} +1$
codewords, then the code is a sunflower.
\end{theorem}

The bound of Theorem~\ref{thm:sunflower_bound} is rather weak compared
to the known lower bounds. The following conjecture
\cite{Shadows}
is attributed to Deza.

\begin{conjecture}
\label{con:largest}
If a $t$-intersecting code in $\cG_q(n,k)$
has more than $\sbinomq{k+1}{1}$ codewords, then the code is a sunflower.
\end{conjecture}

The term $t$-intersecting code is different from the highly related term
$t$-intersecting family~\cite{EKR61,FrWi86}. A family of $k$-subsets
of a set $X$ is called $t$-intersecting if any two $k$-subsets in the
family intersect in at least $t$ elements~\cite{EKR61}. The celebrated
Erd\"{o}s-Ko-Rado theorem determines the maximum size of such family.
The $q$-analog problem was considered for subspaces in~\cite{Hsi75,FrWi86}.
For the $q$-analog problem, a $t$-intersecting family is is a set of $k$-subspaces
whose pairwise intersection is at least $t$. In~\cite{FrWi86} it was proved that

\begin{theorem}
\label{theorem:FW}
If $\mathbb{I} \subseteq \grsmn{q}{n}{k}$ is a family of $k$-subspaces of $\bF_q^n$ whose pairwise intersection is of dimenstion at least $t$, then
$$
|\mathbb{I}| \le \max \left\{ \sbinomq{n-t}{k-t},\sbinomq{2k-t}{k} \right\}.
$$
\end{theorem}

Although Theorem~\ref{theorem:FW}
concerns fundamentally different families from the ones discussed in this paper,
some intriguing resemblance is evident. One can easily see that there exist two
simple families of subspaces attaining the bound of Theorem~\ref{theorem:FW}.
The first is the set of all $k$-subspaces sharing at least a fixed $t$-subspace.
The second is the set of all $k$-subspaces contained in some $(2k-t)$-subspace
These elements are closely related to the terms \emph{sunflower} and \emph{ball}
discussed in details in Section~\ref{sec:trivial}.

One concept which is heavily connected to constant dimension codes
is rank-metric codes.
For two $k \times \ell$ matrices $A$ and $B$ over $\F_q$ the
\emph{rank distance} is defined by
$$
d_R (A,B) = \text{rank}(A-B)~.
$$
A $[k \times \ell,\varrho,\delta]$ \emph{rank-metric code} $\cC$
is a linear code, whose codewords are $k \times \ell$ matrices
over $\F_q$; they form a linear subspace with dimension $\varrho$
of $\F_q^{k \times \ell}$, and for each two distinct codewords $A$
and $B$ we have that $d_R (A,B) \geq \delta$.

There is a large literature on rank-metric code and also on the
connections between rank-metric codes and constant dimension
codes, e.g.~\cite{Del78,EtSi09,Gab85,GaYa10,Rot91,SKK08}. Given
a rank-metric code, one can form from it
a constant dimension code, by lifting its matrices~\cite{EtSi09,SKK08}.
Optimal constant dimension codes can be derived from optimal codes
of a subclass of rank-metric codes, namely \emph{constant rank codes}.
In a constant rank code all matrices have the same rank. The connection
between optimal codes in this class and optimal constant dimension codes
was given in~\cite{GaYa10}. This also can motivate a research on equidistant
constant rank codes which we will also discuss in this paper.

The rest of this paper is organized as follows. In Section~\ref{sec:trivial}
we consider trivial codes and trivial constructions.
The trivial equidistant constant dimension codes are the $q$-analogs
of the trivial binary equidistant constant weight codes. It will be very clear that
the trivial $q$-analogs are not so trivial. In fact for most parameters
we don't know the size of the largest trivial codes. We will also define
two operations which are important in constructions of codes. The first
one is the extension of a code (or the extended code of a given code)
preserves the triviality of the code. The second operation is the
orthogonality (or the orthogonal code of a given code) preserves
triviality only in one important set of parameters.
In Section~\ref{sec:construction} we present our main result of
the paper, a construction of 1-intersecting code in $\cG_q(n,k)$,
$n \geq \binom{k+1}{2}$, whose size is $\sbinomq{k+1}{1}$.
The construction is based on the Pl\"{u}cker embedding~\cite{BR98-PgBook}.
We will show one case in which a larger code by one is obtained,
which falsifies Conjecture~\ref{con:largest} in this case.
We also consider the largest
non-sunflower $t$-intersecting codes in $\cG_q(n,k)$ based
on our discussion. Finally, we
consider the size of the largest equidistant code in all
the projective space $\cP_q(n)$.
In Section~\ref{section:eqRMC} we use the Pl\"{u}cker
embedding to form a constant rank code over $\F_q$, with matrices
of size $n\times\binom{n}{2}$, rank $n-1$, rank distance $n-1$,
and $q^n-1$ codewords. The technique of Section~\ref{section:eqRMC}
is used in Section~\ref{sec:recursion} for a recursive construction of
exactly the same codes constructed in Section~\ref{sec:construction}.
Conclusion and open problems for future research
are given in Section~\ref{sec:conclude}.

\section{Trivial Codes and Trivial Constructions}
\label{sec:trivial}

In this section we will consider first trivial codes.
A binary constant weight equidistant code $\bC$ of size~$m$ is called \emph{trivial}
is every column of $\bC$ has $m$ or $m-1$ equal entries. Now, we define
a $q$-analog for a constant dimension equidistant code.
A constant dimension equidistant code $\C \subset \cG_q(n,k)$
will be called \emph{trivial} if one of the following
two conditions holds.
\begin{enumerate}
\item
Each element $x$ of $\F_q^n$ is either contained in all
$k$-subspaces of $\C$, no $k$-subspace of $\C$, or
exactly one $k$-subspace of $\C$.

\item
If $\T$ is the smallest subspace of $\F_q^n$ which contains
all the $k$-subspaces of $\C$ then $\T$ is a $(k+1)$-subspace.
\end{enumerate}

A trivial code which satisfies the first condition
is a \emph{sunflower}. A trivial code
which satisfies the second condition will be called
a \emph{ball}. We will examine now the types of constant dimension
equidistant codes which satisfy one of these
two conditions.
We will discover that in some cases triviality does not mean
that it is easy to find the optimal (largest) trivial code.

\subsection{Partial spreads}
\label{sec:spread}

Two subspaces $X, Y \in \cG_q(n,k)$ are called \emph{disjoint} if
their intersection is the null space, i.e. ${X \cap Y = \{ {\bf 0} \}}$.
A \emph{partial $k$-spread} (or a \emph{partial spread} in short)
in $\cG_q(n,k)$ is a set of disjoint subspaces from
$\cG_q(n,k)$. If $k$ divides $n$  and the partial spread
has $\frac{q^n-1}{q^k-1}$ subspaces then the partial spread
is called a \emph{$k$-spread} (or a \emph{spread} in short).
A partial spread in $\cG_q(n,k)$ is the $q$-analog of
a 0-intersecting constant weight code
of length $n$ and weight $k$. The
number of $k$-subspaces in the largest partial spread of $\cG_q(n,k)$
will be denoted by $E_q[n,k]$.
The known upper and lower bounds on $E_q[n,k]$
are summarized in the following theorems. The first three well-known
theorems can be found in~\cite{EtVa11}.

\begin{theorem}
\label{thm:exactk_2k}
If $k$ divides $n$ then $E_q[n,k] = \frac{q^n-1}{q^k-1}$.
\end{theorem}

\begin{theorem}
\label{thm:upperk_2k}
$E_q [n,k] \leq \left\lfloor \frac{q^n-1}{q^k-1} \right\rfloor -1$
if $n \not\equiv 0~(\text{mod}~k)$.
\end{theorem}

\begin{theorem}
\label{thm:lowerk_2k}
Let $n \equiv r~(\text{mod}~k)$. Then, for all
$q$, we have
$$
E_q[n,k] \geq  \frac{q^n - q^k(q^r -1)-1}{q^k-1}.
$$
\end{theorem}

The next theorem was proved in~\cite{HoPa72} for $q=2$ and for
any other $q$ in~\cite{Beu75}.

\begin{theorem}
\label{thm:S_exactk_2k}
If $n \equiv 1~(\text{mod}~k)$ then $E_q[n,k] = \frac{q^n-q}{q^k-1} -q+1=\sum_{i=1}^{\frac{n-1}{k}-1} q^{ik+1}+1$.
\end{theorem}

Theorem~\ref{thm:S_exactk_2k} was extended for the case where
$q=2$ and $k=3$ in~\cite{EJSSS} as follows.

\begin{theorem}
If $n \equiv c~(\text{mod}~3)$ then $E_2[n,3] = \frac{2^n-2^c}{7} -c$.
\end{theorem}

The upper bound implied by
Theorem~\ref{thm:S_exactk_2k} was improved for some cases
in~\cite{DrFr79}.

\begin{theorem}
\label{thm:best_U_bound}
If $n=k \ell + c$ with $0 < c < k$, then
$E_q[n,k] \leq \sum_{i=0}^{\ell-1} q^{ik+c} -\Omega -1$,
where $2 \Omega = \sqrt{1+4q^k(q^k-q^c)} -(2q^k-2q^c+1)$.
\end{theorem}

\subsection{Extension of a code}
\label{sec:extension}

An $(n,d,w)$ constant weight equidistant code $\bC$ is \emph{extended}
to an $(n+1,d,w+1)$ constant weight equidistant code $\cE(\bC)$ by
adding a column of \emph{ones} to the code. Similarly,
an $[n,d,k]_q$ constant dimension equidistant code $\C$
is \emph{extended} to an $[n+1,d,k+1]_q$ constant dimension
equidistant code $\cE(\C)$ as follows. We  first define $\F_q^{n+1}$ by
$\F_q^{n+1} = \{ (x,\alpha) ~\vert~ x \in \F_q^n ,~ \alpha \in \F_q \}$.
For a subspace $X \in \cG_q(n,k)$ let $(X,0)$ be a subspace in $\cG_q(n+1,k)$
defined by $(X,0) = \{ (x,0) ~\vert~ x \in X \}$.
Let $v \in \F_q^{n+1} \setminus \{ (x,0) ~\vert~ x \in \F_q^n \}$ and
$\C \subset \cG_q(n,k)$. We define the extended code $\cE(\C)$ by
$$
\cE(\C) = \{ \Span{ (X,0) \cup \{ v \}} ~\vert~ X \in \C \}  ~.
$$
The following theorem can be easily verified.

\begin{theorem}
If $\bC$ and $\C$ are trivial codes, then the extended codes
$\cE(\bC)$ and $\cE(\C)$ are also trivial codes.
\end{theorem}

We note that the extended code $\cE (\C)$ of a trivial
code $\C$ is not unique, but all
such extended codes are isomorphic. An $[n,d,k]_q$ constant dimension
equidistant code $\C$ can be extended $\ell$ times
to an $[n+\ell,d,k+\ell]_q$ constant dimension
equidistant code. This extended code will be denoted by~$\cE^\ell (\C)$.

\subsection{Sunflowers}
\label{sec:sunflower}

A partial spread is clearly
a 0-intersecting sunflower. For a given $n,k$ such that $0 < k<n$ we construct
the largest $t$-intersecting sunflower in $\cG_q(n,k)$ by
using the following two simple theorems.

\begin{theorem}
\label{thm:construct}
If $\bS$ is a $t$-intersecting sunflower in $\cG_q(n,k)$, then
$\cE(\bS)$ is a $(t+1)$-intersecting sunflower in $\cG_q(n+1,k+1)$.
\end{theorem}

\begin{theorem}
\label{thm:proof}
Let $\bS$ be a $t$-intersecting sunflower in $\cG_q(n,k)$ and
let $X$ be an $(n-t)$-subspace of $\F_q^n$ such that
$X \oplus Cen(\bS) =\mathbb{F}_q^n$. Then the set
$\{ X \cap Y ~\vert~ Y \in \bS \}$ is a partial $(k-t)$-spread
in $X$.
\end{theorem}

If $\bS$ is the largest partial $(k-t)$-spread in $\cG_q(n-t,k-t)$,
then by Theorem~\ref{thm:construct} we have that $\cE^t (\bS)$ is
a $t$-intersecting sunflower in $\cG_q(n,k)$. By Theorem~\ref{thm:proof},
we have that
$\cE^t (\bS)$ is the largest $t$-intersecting sunflower in $\cG_q(n,k)$.

\subsection{Optimal $(k-1)$-intersecting equidistant codes in $\cG_q(n,k)$}
In this subsection we present a construction for optimal
$(k-1)$-intersecting equidistant code in $\cG_q(n,k)$ for any $n\ge k+1$.

\begin{theorem}
An optimal non-sunflower $(k-1)$-intersecting equidistant code in
$\cG_q(n,k)$, $n\ge k+1$ has $\qbin{k+1}{k}{q}$ subspaces.
\end{theorem}

\begin{proof}
Let $V$ be any $(k+1)$-subspace of $\bF_q^n$ and let $\C = \qbin{V}{k}{q}$, i.e.
$\C$ consists of all $k$-subspaces of a $(k+1)$- subspace. It is
readily verified that every two such $k$-subspaces intersect
at a $(k-1)$-subspace, and the size of $\C$ is $\qbin{k+1}{k}{q}$.

Let $\C'$ be an equidistant $(k-1)$-intersecting code in $\cG_q(n,k)$.
Let $X \in \C'$, and let $\bS_1,\bS_2$ be any two distinct sunflowers
in $\C'$ such that $Cen(\bS_1),Cen(\bS_2) \sus X$. It is easy to verify
that for every $X_1 \in \bS_1,X_2\in \bS_2$ (which implies that
$\dim(X_1 \cap X_2)=k-1$) we have that $|X_1 \cap X_2 \cap X|=q^{k-2}-1$
which implies that $|X_1 \cap X_2 \cap X^c|=(q^{k-1}-1)-(q^{k-2}-1)=(q-1)q^{k-2}$.

We prove now that for every $X_1\in \bS_1$ and every  $Y,Z \in \bS_2$ the sets
$\cA= X_1\cap Y \cap X^c$ and $\cB= X_1\cap Z \cap X^c$
are mutually disjoint. Since $Y,Z \in \bS_2$ and $Cen(\bS_2) \sus X$,
it follows that $Y$ and $Z$ do not intersect outside $X$,
i.e. $(Y\cap Z)\cap X^c = \varnothing$. If $\cA\cap \cB \ne \varnothing$
or equivalently $X_1\cap Y \cap Z \cap X^c \ne \varnothing$ then
$Y \cap Z \cap X^c \ne \varnothing$, a contradiction.

Therefore, in the set $\{ X_1 \cap Y \cap X^c ~\vert~ Y \in \bS_2 \}$
there can be at most $(q^k-q^{k-1})/(q-1)q^{k-2}=q$
disjoint subsets of size $(q-1)q^{k-2}$. Hence, the size of any sunflower
other than $\bS_1$ is at most $q+1$. However, all of the above arguments
are applicable for any initial sunflower $\bS_1$. Therefore, each sunflower
whose center is inside $X$ have at most $q+1$ codewords, including $X$ itself.
$X$ may have at most $\qbin{k}{k-1}{q}$ sunflower centres
inside it, which yields that:
\[
|\C'| \le 1+q\cdot \qbin{k}{k-1}{q} = \qbin{k+1}{k}{q} = |\C|~.
\]
Thus, the claim in the theorem follows.
\end{proof}

\begin{corollary}
An optimal non-sunflower equidistant $(k-1)$-intersecting code in
$\cG_q(n,k)$, $n\ge k+1$ consists of all $k$-subspaces of any
given $(k+1)$-subspace. This code is a trivial code which
satisfies the second condition, i.e., it is a ball.
\end{corollary}

\subsection{Orthogonal subspaces}

A simple operation which preserves the equidistant property
of a code and its triviality in the Hamming scheme is the \emph{complement}.
Two binary words $x = (x_1, \ldots , x_n)$
and $y = (y_1, \ldots , y_n)$ are \emph{complements}
if for each $i$, $1 \leq i \leq n$, we have $x_i + y_i = 1$, i.e.,
$x_i = 0$ if and only if $y_i = 1$. We say that $y$ is the complement
of $x$ and denote it by $y=\bar{x}$. For a binary code $\bC$, the
\emph{complement} of $\bC$, $\bar{\bC}$ is defined by
$\bar{\bC} = \{ x ~\vert~ \bar{x} \in \bC \}$.
It is easily verified that if a constant weight code $\bC$
is equidistant then also its complement is constant weight
code and equidistant and if $\bC$ is trivial then also its
complement $\bar{\bC}$ is trivial.

What is the $q$-analog operation for the complement? This question
was discussed in details before~\cite{BrEtVa13}. We will use the
orthogonality as the the $q$-analog for complement.
For a constant dimension code $\C \sus \cG_q(n,k)$ we define
the \emph{orthogonal} code by
$\C^{\bot} = \{X \vert X^{\bot} \in \C\}$.
It is known~\cite{EtVa11,KoKs08,XiFu09} that for any two subspaces
$X,Y \in \F_q^n$ we have $d_S(X,Y)=d_S (X^\perp,Y^\perp)$.
This immediately implies the following result.

\begin{theorem}
If $\C$ is a $t$-intersecting equidistant code in $\cG_q(n,k)$,
then $\C^\perp$ is a $t'$-intersecting equidistant code in
$\cG_q(n,n-k)$, where $t'=n-2k+t$.
\end{theorem}
\begin{proof}
Clearly, $\C^\perp$ is a code in $\cG_q(n,n-k)$. Furthermore,
$\C$ and $\C^\perp$ have the same minimum subspace
distance $d$ which is also the distance
between any two codeword of $\C$ and between any two codewords
of $\C^\perp$. Therefore, $d=2(k-t)=2(n-k-t')$ which implies
$t'=n-2k+t$.
\end{proof}

\begin{corollary}
If $\C$ a 0-intersecting code in $\cG_q(n,k)$ (a partial spread),
then $\C^\perp$ is an $(n-2k)$-intersecting code in
$\cG_q(n,n-k)$.
\end{corollary}

\begin{corollary}
The smallest possible $n$ for which a $t$-intersecting code exists
in $\cG_q(n,k)$ is $n=2k-t$. The size of the largest such code is the
size of the largest partial spread in $\cG_q(2k-t,k-t)$.
\end{corollary}

The next question we would like to answer is whether the orthogonal
code of a trivial code is also a trivial code. The answer is that usually,
the orthogonal code of a trivial code is not a trivial code. The only exception
is given in the following theorem whose proof is easy to verify.

\begin{theorem}
$\C$ is an optimal $(k-1)$-intersecting equidistant code in $\cG_q(n,k)$ (a ball)
if and only if~$\C^\perp$
is a $(n-k-1)$-intersecting sunflower in $\cG_q(n,n-k)$.
\end{theorem}

\section{Constructions of Large Non-Sunflower Equidistant Codes}
\label{sec:construction}

In this section, we will consider the
size of the largest non-sunflower equidistant code in $\cG_q(n,k)$
and in $\cP_q(n)$. The main result which will be presented in
subsection~\ref{sec:Plucker} is a construction of 1-intersecting
codes in $\cG_q(n,k)$,
$n \geq \binom{k+1}{2}$, whose size is $\sbinomq{k+1}{1}$.
This construction will be based on
the Pl\"{u}cker embedding. In subsection~\ref{sec:larger}
we will show the only example we have found (by computer search)
for which this construction is not optimal.
In subsection~\ref{sec:the_largest} we will consider the largest
non-sunflower $t$-intersecting codes in $\cG_q(n,k)$ based
on our discussion. In subsection~\ref{sec:projective} we will
consider the size of the largest equidistant code in the whole
the projective space $\cP_q(n)$.

\subsection{Construction using the Pl\"{u}cker embedding}
\label{sec:Plucker}

Let $[n]$ denote the set $\{1,2,\ldots,n\}$, let $e_i$ denote
the unit vector (of a given length) with a 1 in the
$i$th coordinate, and
for any $0 < k\le n$, let $n_k= {n \choose k}$.
We denote by $\bP_q^{\ell-1}$ the set of \textit{projective points}
of $\bF_q^\ell$. The set $\bP_q^{\ell-1}$
is commonly referred to as the set $\nicefrac{\bF_q^{\ell}}{\sim}$
where $\sim$ is an equivalence relation over $\bF_q^\ell \setminus \{0\}$ defined as
\[
x\sim y \iff \exists \lambda \in \bF_q\setminus \{0\},x=\lambda y ~.
\]
It is widely
known~\cite{BR98-PgBook,RST13-onTheGeometry,RT12-Decoding} that $\cG_q(n,k)$ can be
embedded in $\bP_q^{n_k-1}$ using the Pl\"{u}cker embedding,
denoted by $P$.

Given $U\in \cG_q(n,k)$, let $M(U)\in \bF_q^{k\times n}$
be some $k\times n$ matrix over $\bF_q$ whose row span is $U$.
Given any $k$-subset $\{i_1,\ldots,i_k\}$ of $[n]$,
let $M(U)(i_1,\ldots,i_k)$ be the $k\times k$ sub-matrix of $M(U)$
consisting of columns $i_1,\ldots,i_k$. Consider the $n \choose k$
coordinates of $\bF_q^{n_k}$ as numbered by $k$-subsets of $[n]$,
and for each $k$-subset $\{i_1,\ldots,i_k\}$ of $[n]$ assume
w.l.o.g that $i_1<\ldots<i_k$. The function $P$ maps
a subspace $U\in \cG_q(n,k)$ to the equivalence
class of the vector $v(U)\in \bF_q^{n_k}$, defined as:
\[
(v(U))_{\{i_1,\ldots,i_k\}} = \det M(U)(i_1,\ldots,i_k).
\]
Namely, the coordinate $\{i_1,\ldots,i_k\}$ of the vector $v(U)$
is the determinant of the sub-matrix of $M(U)$ consisting
of columns $i_1<\ldots<i_k$ of $M(U)$. Formally, the
function $P$ is defined as $P(U) = [v(U)]$, where $[v(U)]$
is the equivalence class of the vector $v(U)$ under the
equivalence relation $\sim$ defined earlier in this section.
It is worth mentioning that the function $P$ is well-defined,
as any choice of a matrix $M(U)$ whose row span is $U$ will result
in the same equivalence class in $\bP_q^{n_k-1}$~\cite{BR98-PgBook}. In this paper,
in order to maintain consistency with coding theory terminology,
we identify the set $\bP_q^{\ell-1}$ by the set of
all 1-subspaces of $\F_q^{\ell}$, namely, $\cG_q(\ell,1)$.

Another concept we use in this section is Steiner systems.
A \emph{Steiner system} $S(t,k,n)$ is a pair $(Q,B)$ where $Q$ is
an $n$-set of elements (called \emph{points}) and $B$ is a collection
of $k$-subsets of $Q$ (called \emph{blocks}), such that every
$t$-subset of $Q$ is contained in exactly one block of $B$.
A Steiner system can be described by its incidence matrix. This
is a matrix $A=(a_{ij}),i\in [n],j\in [b], b= |S(t,k,n)|$,
where if $Q=\{q_1,\ldots,q_n\}$ and $B=\{B_1,\ldots,B_b\}$ we have:
\[
a_{ij} = \left\{ \begin{array}{ll}
         1 & \mbox{if $q_i \in B_j$}\\
         0 & \mbox{if $q_i \notin B_j$}\end{array} \right ..
\]

The following lemma, which is a key in the construction which follows, can be easily verified.

\begin{lemma} \label{lemma:KeyLemma}
The rows of the matrix $A$ defined by a Steiner system $S(2,k,n)$
form an 1-intersecting equidistance code.
The weight of each codeword in this code is $\frac{n-1}{k-1}$.
\end{lemma}

In what follows, we use the Pl\"{u}cker embedding, together with
Lemma~\ref{lemma:KeyLemma} to construct equidistant constant dimension codes.

\begin{theorem} \label{thm:PluckerConstruction}
For every integer $n \geq 3$, there exists an
$1$-intersecting equidistant code in $\cG_q(n_2,n-1)$ of size~$\qbin{n}{1}{q}$, where $n_2={n\choose 2}$.
\end{theorem}

The main idea of the construction in the proof of
Theorem~\ref{thm:PluckerConstruction} is to consider
a set (design) $\cS_{q,n}$ of blocks, whose blocks are the
elements of $\cG_q(n,2)$ and its set of points
is $\cG_q(n,1)$ where a point $X\in \cG_q(n,1)$
is incident with a block $Y\in \cG_q (n,2)$
if $X\sus Y$. $\cS_{q,n}$ forms a Steiner system
$S(2,q+1,\frac{q^{n}-1}{q-1})$ since every 2-subspace
contains $q+1$ points and every pair of
distinct 1-subspaces is contained in a unique 2-subspace.
We embed $\cG_q (n,2)$ into $\bP_q^{n_2-1}$ using
the Pl\"{u}cker embedding, and show that given $V \in \cG_q(n,1)$,
the set
$$
P_V = \bigcup_{\substack{{U\in \cG_q (n,2)} \\ {V\sus U}}} P(U)
$$
is the union of the Pl\"{u}cker embeddings
of subspaces in $\cG_q (n,2)$ which intersect
at $V$ and constitutes a subspace
in $\cG_q (n_2,n-1)$. Therefore, we may use
Lemma \ref{lemma:KeyLemma} to get an
1-intersecting equidistant code in $\cG_q (n_2,n-1)$ of
size $\qbin{n}{1}{q}=\frac{q^n -1}{q-1}$.
The proof of Theorem~\ref{thm:PluckerConstruction} relies on the following two lemmas.
\begin{lemma}\label{lemma:sizeOfPV}
$|P_V|=q^{n-1}$.
\end{lemma}

\begin{proof}
Note that the number of 2-subspaces that contain a given 1-subspace is $\frac{q^{n-1}-1}{q-1}$. Since each two distinct 1-subspaces $P(U_1),P(U_2)$ itersect trivially, it follows that $|P_V| = \frac{q^{n-1}-1}{q-1}\cdot (q-1) +1=q^{n-1}$.
\end{proof}

\begin{lemma} \label{lemma:PluckerAreVS}
If $V\in \cG_q (n,1)$ then $P_V \in \cG_q (n_2,n-1)$.
\end{lemma}

\begin{proof}
Let $v$ be an arbitrary nonzero vector in $V$. Let $r$ be an arbitrary
index such that $v_r \ne 0$ and let $\cB_v = \{z^1,\ldots,z^{n-1}\}$ be the set of $n-1$ distinct unit vectors of
length $n$ such that $e_r \notin \cB_v$.
If $Z_i = \Span{\{v,z^i\}},1\le i\le n-1,$ then by the definition
of Pl\"{u}cker embedding, the projective point $P(Z_i)$, considered
as an 1-subspace of $\bF_q^{n_2}$, is the span of the vector $p^i$ of length $n_2$, whose coordinates are indexed by the subsets $\{s,t\}\subseteq[n]$ and are defined by
\begin{eqnarray}
p^i_{\{s,t\}}= \det_{(s,t)}{v \choose z^i}=v_s z^i_t-v_t z^i_s \label{eqn:DefnPi}.
\end{eqnarray}
Recall that the coordinates of~$\bF_q^{n_2}$ are identified
with 2-subsets of $[n]$ as mentioned earlier in this section.

We will now prove that $\Span{\{p^1,\ldots,p^{n-1}\}}=P_V$,
and since $|P_V| = q^{n-1}$ by Lemma \ref{lemma:sizeOfPV}, it implies that $P_V$ is an $(n-1)$-subspace of~$\bF_q^{n_2}$.

Let $x= \sum_{i\in [n-1]} a_ip^i$ where $a_i\in \bF_q$ for all $i\in[n-1]$.
We will show that $x\in P_V$, i.e. that there exists
a $\tilde{Z} \in \cG_q (n_2 ,2)$ such that $V\sus \tilde{Z}$
and $x \in P(\tilde{Z})$. As a consequence, we will have that $\Span{\{p^1,\ldots,p^{n-1}\}}\sus P_V$.
Let
\begin{align}
\tilde{z} & =  \sum_{i\in [n-1]} a_iz^i. \label{eqn:DefnWTilde}
\end{align}
Clearly, $V\sus \Span{\{v,\tilde{z}\}}$, and since
$\tilde{z}_r = 0,v_r \ne 0$ it follows that $\Span{\{v,\tilde{z}\}}$
is a 2-subspace of $\bF_q^n$. By the definition of
the Pl\"{u}cker embedding, $P(\Span{\{v,\tilde{z}\}})$ is
the span of the vector $\tilde{x}$, whose coordinates are indexed by the subsets $\{s,t\}\subseteq[n]$ and are defined by
\begin{align*}
\tilde{x}_{\{s,t\}}& = \det _{(s,t)} {v \choose \tilde{z}}\\
&=v_s \tilde{z}_t - v_t \tilde{z}_s \\
&=v_s \left(\sum _{i\in[n-1]}a_i z^i_t\right) - v_t \left(\sum_{i\in [n-1]} a_i z^i_s\right)&\text{by (\ref{eqn:DefnWTilde})} \\
&= \sum_{i\in [n-1]}a_i\left(v_s z^i_t-v_t z^i_s\right)\\
&= \sum_{i\in[n-1]}a_i p^i_{\{s,t\}}&\text{by (\ref{eqn:DefnPi})}\\
&=x_{\{s,t\}}&\text{by the definition of $x$}
\end{align*}
Hence $\tilde{x}=x$ and therefore $\Span{\{p^1,\ldots,p^{n-1}\}}\sus P_V$.

Now, let $x \in P_V$ and let $U_x\in \cG_q(n,2)$ be the
unique $2$-subspace such that $V \sus U_x$ and $x\in P(U_x)$.
Hence, $U_x =  \Span{\{v,z\}}$ for some $z\in \bF_q^n$. By the definition
of $\cB_v = \{z^1,\ldots,z^{n-1}\}$ every vector whose $r$th entry is 0 is in~$\Span{\cB_v}$.
If $z_r \ne 0$ then we define
\[
z' = v-\left(\frac{v_r}{z_r}\right)\cdot z\in U_x.
\]
Clearly, $z'_r=0$ and hence $U_x = \Span{\{v,z'\}}$. Thus, w.l.o.g we assume
$z_r=0$ and we may write ${z = \sum _{i\in [n-1]} a_i z^i}$ for
some $a_i \in \bF_q$, $i\in[n-1]$. Since
$U_x = \Span{\{v,z\}}$ it follows that the 1-subspace $P(U_x)$ is spanned
by the vector $\left( \det _{(s,t)}{v \choose z}\right) _{\{s,t\}\sus[n]}$.
Hence, there exists some $\lambda \in \bF_q$ such that:
\begin{align*}
x_{\{s,t\}} & =  \lambda \cdot \det _{(s,t)} {v \choose z}=\lambda (v_s z_{t}- v_t z_{s}) \\
&= \lambda \left( v_s \cdot \sum_{i\in [n-1]}a_i z^i_t-v_t \cdot \sum_{i\in [n-1]}a_i z^i_s\right)&\text{since $z = \sum _{i\in [n-1]} a_i z^i$}\\
&= \lambda  \sum_{i\in [n-1]}a_i \left(v_s z^i_t-v_t z^i_s \right)\\
&=\lambda \sum_{i\in [n-1]}a_i p^i_{\{s,t\}}&\text{by (\ref{eqn:DefnPi})}
\end{align*}
Therefore $x \in \Span{\{p^1,\ldots,p^{n-1}\}}$ which implies that $P_V\sus \Span{\{p^1,\ldots,p^{n-1}\}}$. Thus we have proved that $P_V = \Span{\{p^1,\ldots,p^{n-1}\}}$, and as a consequence we have that ${P_V \in \cG_q (n_2 ,n-1)}$.
\end{proof}

\begin{proof}
(of Theorem \ref{thm:PluckerConstruction}) Let $\C\sus \cG_q(n_2 ,n-1)$ be the code defined by
\[
\C = \{P_V ~\vert~ V\in \cG_q(n,1) \}.
\]
By Lemma~\ref{lemma:PluckerAreVS}, for each
$V \in \cG_q(n,1)$ we have that $P_V$ is an $(n-1)$-subspace
of $\bF_q^{n_2}$ and hence $\C$ is well-defined.
By Lemma~\ref{lemma:KeyLemma} and the discussion on $\cS_{q,n}$,
it follows that $\C$ is an 1-intersecting equidistant code.
\end{proof}

\begin{remark}
After the paper was written, we found that Lemma~\ref{lemma:PluckerAreVS}
can also be obtained as a consequence of~\cite[Theorem 24.2.9, p.113]{hir91}
which discuss the theory of finite projective geometries. But, the proof of
this Theorem requires more detailed theory which precedes it, while our
proof is much shorter, simpler, and direct.
\end{remark}

\subsection{Do larger equidistant codes exist?}
\label{sec:larger}

It was believed (Conjecture~\ref{con:largest}) that the largest
non-sunflower 1-intersecting code in $\cG_q(n,k)$ has size at most
$\sbinomq{k+1}{1}$. The following example consists
of a non-sunflower $1$-intersecting
code $\C \sus \cG_2(6,3)$ of size 16, while $\sbinomtwo{4}{1}=15$.
$\C$ was found by a computer search.

Let $\alpha$ be a primitive root of $x^6+x+1$, and use this
primitive polynomial to generate $\bF_2^6$. Let $\C$ be the
code which consists of the following sixteen 3-subspaces:

\begin{equation*}
\begin{split}
&\Span{\{\ga^{0},\ga^{1},\ga^{2}\}}\\
&\Span{\{\ga^{0},\ga^{15},\ga^{10}\}}\\
&\Span{\{\ga^{6},\ga^{52},\ga^{51}\}}\\
&\Span{\{\ga^{12},\ga^{54},\ga^{15}\}}\\
\end{split}
\hspace{15 mm}
\begin{split}
&\Span{\{\ga^{10},\ga^{26},\ga^{25}\}}\\
&\Span{\{\ga^{18},\ga^{1},\ga^{59}\}}\\
&\Span{\{\ga^{33},\ga^{20},\ga^{59}\}}\\
&\Span{\{\ga^{49},\ga^{26},\ga^{46}\}}\\
\end{split}
\hspace{15 mm}
\begin{split}
&\Span{\{\ga^{12},\ga^{9},\ga^{29}\}}\\
&\Span{\{\ga^{25},\ga^{0},\ga^{58}\}}\\
&\Span{\{\ga^{41},\ga^{2},\ga^{36}\}}\\
&\Span{\{\ga^{36},\ga^{34},\ga^{30}\}}\\
\end{split}
\hspace{15 mm}
\begin{split}
&\Span{\{\ga^{6},\ga^{5},\ga^{33}\}}\\
&\Span{\{\ga^{19},\ga^{10},\ga^{6}\}}\\
&\Span{\{\ga^{58},\ga^{6},\ga^{49}\}}\\
&\Span{\{\ga^{36},\ga^{29},\ga^{26}\}}\\
\end{split}
\end{equation*}

\subsection{Large non-sunflower equidistant codes}
\label{sec:the_largest}

In this subsection, we will consider the construction of the
largest $t$-intersecting codes in $\cG_q(n,k)$. For~$n$ large
enough this code is a sunflower and hence for such large $n$ we will
consider also the largest $t$-intersecting code which is not a sunflower.

By Theorem~\ref{thm:sunflower_bound}, sunflowers are the
largest constant dimension equidistant codes when the ambient
space is large enough. The size of the largest sunflower is
usually not known, but we know that it is equal to the
size of a related partial spread. Therefore, we would like to
know what is the size of the largest $t$-intersecting code
in $\cG_q(n,k)$ which is not a sunflower.

Assume we want to generate a $(k-r)$-intersecting code
in $\cG_q(n,k)$. Clearly we must have $n \geq k+r$.
If $k-r=0$ then any $(k-r)$-intersecting code is a partial spread
and hence also a sunflower. Therefore, we assume that
$k-r >0$.
We start with the largest partial spread $\bS$ in $\cG_q(k+r,r)$.
By Theorem~\ref{thm:lowerk_2k}, its size $m$ is at least $q^k +1$.
$\bS^\perp$ is a non-sunflower $(k-r)$-intersecting code in
$\cG_q(k+r,k)$ whose size~$m$
is at least $q^k+1$. If $k-r > 1$ then we don't know
how to construct a larger code. If $k-r=1$ then larger codes
of size $\frac{q^{k+1}-1}{q-1}$
are constructed in subsection~\ref{sec:Plucker}.

\subsection{Equidistant codes in $\cP_q(n)$}
\label{sec:projective}

So far we have considered only constant dimension equidistant codes.
Can we get larger unrestricted subspace equidistant codes over $\F_q^n$ than
constant dimension equidistant codes over $\F_q^n$? We start by
considering first equidistant codes in the Hamming scheme.
Let $B_q(n,d)$ be the maximum size of an equidistant code of length~$n$
and minimum Hamming distance $d$ over $\bF_q$. Let $B_q(n,d,w)$ be
the maximum size of an equidistant code of length $n$,
constant weight $w$, and minimum Hamming distance $d$.
The following result, due to~\cite{Wei03-OnEquidistant} shows that when
discussing equidistant codes in the Hamming scheme,
we may restrict our attention to constant weight codes:

\begin{theorem}
\label{thm:equiequi}
$B_q(n,d)=1+B_q(n,d,d)$.
\end{theorem}

A related $q$-analog theorem might hold in some cases, but generally
it does not hold as demonstrated in the following example. The example
is specific in some sense, but it can be generalized to many other parameters.

Let $n$ be an odd integer for which the largest 2-intersecting equidistant code
in $\cG_2(n,4)$ with minimum subspace distance 4 is a sunflower. The size of
the largest partial 2-spread in $\cG_2 (n-2,2)$ is $\frac{2^{n-2}-4}{3}$. Let $\C$ be
such a partial spread. Clearly, $\cE^2 (\C)$ is the largest 2-intersecting
equidistant code in $\cG_2(n,4)$. Let $x$, $y$, $z$, and $u$ be the only nonzero vectors
of $\F_2^{n-2}$ which do not appear in any 2-subspace of $\C$.
Let $v_1 = ( {\bf 0} , 01)$, $v_2 = ( {\bf 0} , 10)$ be two vectors
in $\F_2^n$ and let ${\cE^2 (\C) = \{ \Span{ (X,00) \cup \{ v_1,v_2 \} } ~\vert~ X \in \C \}}$,
where ${(X,00) = \{ (x,00) ~\vert~ x \in X \}}$.
The code
$$
\C' = \cE^2 (\C) \bigcup \{ \Span{ ( {\bf 0} , 01), (y,11) } , \Span{ ( {\bf 0} , 10), (z,11) } , \Span{ ( {\bf 0} , 11), (x,11) } \}
$$
is an equidistant code in $\cP_2(n)$ whose size is $\frac{2^{n-2}+5}{3}$
and its subspace distance is 4. This code is larger than $\cE^2 (\C)$,
which implies that $q$-analog of Theorem~\ref{thm:equiequi} does not exist
in general.

%

\section{Equidistant Rank Metric Codes}
\label{section:eqRMC}

In this section, we present a connection between the construction
presented in Section \ref{sec:Plucker} and equidistant rank-metric codes.
To the best of our knowledge, this is the first construction of
an equidistant rank metric code whose matrices are not of full rank.

In this section we use a \textit{variant} of the function $P$,
defined in Section \ref{sec:Plucker}, denoted by $\varphi$.
This variant may be considered as acting on matrices from
$\bF_{q}^{k\times n}$ rather than on $\grsmn{q}{n}{k}$,
and maps them to $\bF_{q}^{n_k}$ rather than to $\bP_q^{n_k-1}$.

\begin{definition}
Given $M\in \bF_q^{k\times n}$, identify the coordinates
of $\bF_q^{n_k}$ with $k$-subsets of $[n]$, and define
$\varphi(M)$ as a vector of length $n_k = {n \choose k}$ with:
\[
\left( \varphi(M) \right) _{\left(i_1,\ldots,i_k\right)} = \det M\left( i_1 ,\ldots,i_k\right)
\]
where $M\left( i_1 ,\ldots,i_k\right)$ is the $k\times k$ sub-matrix
of $M$ formed from columns $i_1 ,\ldots,i_k$ of $M$.
\end{definition}

For $v,u \in \bF_q^n$ denote $X_{v,u}= \varphi {v \choose u}$. For $v\in \bF_q^n \setminus\{0\}$, define:
\[
M_v =
\begin{pmatrix}
X_{v,e_1} \\
\vdots \\
X_{v,e_n}
\end{pmatrix},
\]
where $e_i$ is a unit vector of length $n$. Let
$\mathbb{M} = \Span{\{M_{e_1},\ldots,M_{e_n}\}}$.
In the rest of this section we prove
that $\mathbb{M}\setminus\{0\}$ is
an equidistant constant rank code of size $q^n-1$.

The following lemma shows that for $k=2$, the function
$\varphi$ is \textit{linear} in some sense, when one
of the vectors in the matrix it operates on is fixed.

\begin{lemma}
\label{lemma:PluckerIsLinear}
If $u,v,w \in \bF_q^n$ and $\alpha,\beta \in \bF_q$ then
$ X_{v,\alpha u+\beta w} = \alpha \cdot X_{v,u}+\beta\cdot X_{v,w}$.
Similarly, $X_{\alpha u+\beta w,v} = \alpha \cdot X_{u,v}+\beta\cdot X_{w,v}$.
\end{lemma}

\begin{proof}
Consider the $\{s,t\}$ coordinate of the vector $\alpha \cdot X_{v,u}+\beta\cdot X_{v,w}$:
\begin{eqnarray*}
\left(\alpha \cdot X_{v,u}+\beta\cdot X_{v,w}\right)_{\{s,t\}} &=& \alpha\left(v_s u_t - v_t u_s\right) +\beta \left( v_s w_t - v_t w_s\right)\\
&=& v_s\left(\alpha u_t + \beta w_t\right)-v_t\left(\alpha u_s + \beta w_s\right)\\
&=& v_s \left(\alpha u + \beta w\right)_t - v_t\left( \alpha u + \beta w\right)_s\\
&=& \left(\varphi {v \choose \alpha u +\beta w}\right)_{\{s,t\}}\\
&=& \left(X_{v,\alpha u+\beta w}\right)_{\{s,t\}}.
\end{eqnarray*}
The claim that $X_{\alpha u+\beta w,v} = \alpha \cdot X_{u,v}+\beta\cdot X_{w,v}$ has a similar proof.
\end{proof}

\begin{corollary}\label{corollary:basisOfCodewords}
If $V\in\grsmn{q}{n}{1}$ and $v\in V\setminus \{0\}$ then $P_V = \Span{\{X_{v,z^1},\ldots,X_{v,z^{n-1}}\}}$, where $z^i,i\in[n-1]$ was defined in Lemma \ref{lemma:PluckerAreVS}.
\end{corollary}

We now show that each nonzero codeword in $\mathbb{M}$
can be written as $M_v$ for some $v\in \bF_q^n$, where
$rank(M_v)=n-1$. Hence, the linearity of the code implies
that the rank of the difference between any two matrices
in the code is $n-1$, and therefore the code is equidistant.

\begin{lemma} \label{lemma:codeMrepresentation}
$\mathbb{M}=\left\{ M_v \vert v\in \bF_q^n \right\}$.
\end{lemma}

\begin{proof}
Let $M = \sum_{i=1}^n \alpha_i M_{e_i} \in \mathbb{M}$, for some $\alpha_i \in \bF_q, i\in [n]$. By Lemma \ref{lemma:PluckerIsLinear}, the $j$th row of $M$ is:
\begin{eqnarray*}
\sum_{i=1}^n \alpha_i X_{e_i,e_j} =  X_{\left(\sum_{i=1}^n \alpha _i e_i\right),e_j}= X_{v,e_j}~,
\end{eqnarray*}
where $v=\sum_{i=1}^n \alpha_i e_i$, and therefore $M=M_v$.
Conversely, let $v = \sum_{i=1}^n \alpha_i e_i$.
The same arguments (in reversed order) shows
that $M_v = \sum_{i=1}^n \alpha_i M_{e_i}$ and hence $M_v \in \mathbb{M}$.
\end{proof}

\begin{lemma}
If $v = \sum_{i=1}^n \alpha_i e_i, ~\alpha_i\in \bF_q,i\in[n]$
is a non-zero vector in $\bF_q^n$ then $rank(M_v)=n-1$.
\end{lemma}

\begin{proof}
By Corollary~\ref{corollary:basisOfCodewords}, the rows of $M_v$
contain a basis for the codeword $P_{\Span{v}}$ (following the
notations of Theorem~\ref{thm:PluckerConstruction}). Recall
(see Lemma~\ref{lemma:PluckerAreVS}) that a basis for $P_\Span{v}$
is obtained by omitting one vector from the set
$\left\{X_{v,e_1},\ldots,X_{v,e_n}\right\}$. Note that $r$ is
the index of the omitted vector and $\alpha_r\ne 0$
(see Lemma~\ref{lemma:PluckerAreVS}). To complete the proof,
we have to show that the vector $X_{v,e_r}$ lies inside $P_\Span{v}$.
Notice that by Lemma~\ref{lemma:PluckerAreVS}
we have that $\alpha_r \ne 0$. We have that
\begin{eqnarray*}
\sum_{i\in[n]\setminus\{r\}}\alpha_i X_{v,e_i}&=&\varphi {v \choose \sum_{i\in[n]\setminus\{r\}}\alpha_i e_i}\\
&=& \varphi { v \choose v-\alpha_r e_r}\\
&=& \varphi { v \choose -\alpha_r e_r}\\
&=&-\alpha_r \varphi { v \choose e_r}=-\alpha_r X_{v,e_r}.
\end{eqnarray*}
Therefore, $X_{v,e_r} = -\frac{1}{\alpha_r}\sum_{i\in [n]\setminus\{r\}}\alpha_i X_{v,e_i}$. Thus, $\mbox{rowspan}\left(M_v\right) = P_\Span{v}$ and $rank\left(M_v\right)=\dim P_\Span{v}=n-1$.
\end{proof}

\begin{lemma}
If $M_v = M_u$ for $u,v\in\bF_q^n$ then $u=v$.
\end{lemma}

\begin{proof}
If $M_u=M_v$ then $X_{v,e_j}=X_{u,e_j}$ for all $j\in[n]$, which implies
\begin{eqnarray} \label{equation:codeMdim}
\forall j\in[n],\forall \{s,t\} \in {[n] \choose 2},~v_s e_{j,t}-v_t e_{j,s}=u_s e_{j,t}-u_t e_{j,s}.
\end{eqnarray}
In particular, for any $i\in [n]\setminus\{j\}$ we may
choose $s=i,t=j$ and obtain that $v_i = u_i$. Since (\ref{equation:codeMdim})
holds for any $j\in[n]$, we also have $v_j=u_j$ and thus $u=v$.
\end{proof}

\begin{corollary}
$|\bM|=q^n$.
\end{corollary}

\begin{corollary}
$\bM\setminus \{0\}$ is an equidistant constant rank code
over $\Fq$ with matrices of size $n\times {n \choose 2}$, rank $n-1$, rank distance $n-1$, and size $q^n-1$.
\end{corollary}

There is some similarity between the code $\bM$ and the
Sylvester's type Hadamard matrix of order $2^n$~\cite{McSl}.
This code is a binary linear of dimension $n$ and length $2^n$.
Each codeword, except for the all zeros codeword has weight $2^{n-1}$
and the mutual Hamming distance between any two codewords in $2^{n-1}$.
The analog to the code $\bM$ seems to be obvious.

\section{Recursive Construction of Equidistant Subspace Codes}
\label{sec:recursion}

We have shown in Section \ref{sec:Plucker} a direct
construction of an 1-intersecting code in $\grsmn{q}{n_2}{n-1}$.
In this section, we prove that this code can be constructed recursively,
by using some of the results of Section \ref{section:eqRMC}.

Let $\mathbb{C}_{n-1},\mathbb{C}_{n}$ be the 1-intersecting
codes in $\grsmn{q}{{n-1 \choose 2}}{n-2},\grsmn{q}{{n \choose 2}}{n-1}$,
respectively, as constructed in Section \ref{sec:Plucker}.
We first present a construction of \textit{some} code
$\mathbb{D}\subseteq\grsmn{q}{{n \choose 2}}{n-1}$ from
$\mathbb{C}_{n-1}$ and later prove that $\mathbb{D}=\mathbb{C}_n$.

Let $\hat{v}\in\bF_q^{n-1}\setminus\{0\}$. For the purpose of the construction, let $X_{\hat{v},\hat{e}_i}=\varphi{\hat{v}\choose \hat{e}_i}$ as in Section \ref{section:eqRMC} ($\hat{e}_i$ is a unit vector of length $n-1$), and let $\cB_{\hat{v}} = \left\{\hat{z}^1,\ldots,\hat{z}^{n-2}\right\}$ be the set of $n-2$ unit vectors of length $n-1$ such that $P_\Span{\hat{v}} = \Span{\left\{X_{\hat{v},\hat{z}^1},X_{\hat{v},\hat{z}^2},\ldots,X_{\hat{v},\hat{z}^{n-2}}\right\}}$, as denoted in the proof of Lemma \ref{lemma:PluckerAreVS}. For $v\in\bF_q^n\setminus\{0\}$ we define $e_i,X_{v,e_i}$ and $\cB_v$ similarly.

For each codeword $P_\Span{\hat{v}} \in \mathbb{C}_{n-1}$ we construct $q$ codewords in $\mathbb{D}$, denoted by $\left\{U_{\hat{v},a}\right\}_{a\in\bF_q}$, as follows:
\begin{eqnarray*}
U_{\hat{v},0} &= & \mbox{rowspan}
\begin{pmatrix}
\hat{v} & \textbf{0} \\
\textbf{0} & X_{\hat{v},\hat{z}^1}\\
\vdots & \vdots \\
\textbf{0} & X_{\hat{v},\hat{z}^{n-2}}
\end{pmatrix} \\
\forall a \ne 0, U_{\hat{v},a} &= & \mbox{rowspan}
\begin{pmatrix} [ccc|c]
  &               &  & a\cdot X_{\hat{v},\hat{e}_1}\\
  & I_{(n-1)\times (n-1)} &  & \vdots\\
  &               &  & a\cdot X_{\hat{v},\hat{e}_{n-1}} ~.
\end{pmatrix}
\end{eqnarray*}
In addition, we add a codeword $U_0 = \mbox{rowspan}\left(I_{(n-1)\times (n-1)}\vert \textbf{0}\right)$ to $\bD$.

\begin{theorem}
$\bD = \mathbb{C}_{n}$.
\end{theorem}

\begin{proof}
We prove that any $P_\Span{v} \in \mathbb{C}_n$ is equal
to some codeword in $\mathbb{D}$. The equality, $\mathbb{D}=\mathbb{C}_n$, will follow
since $|\mathbb{D}|\le |\mathbb{C}_{n-1}|\cdot q+1=\qbin{n-1}{1}{q}\cdot q+1=\qbin{n}{1}{q} = |\mathbb{C}_n|$.
Let $P_\Span{v} \in \mathbb{C}_n$ for $v\in\bF_q^n\setminus\{0\}$, and let $v=(a,\hat{u})\ne(0,\textbf{0})$,
where $a\in\bF_q,\hat{u}\in\bF_q^{n-1}$. To find the related
codeword in $\mathbb{D}$, we distinguish between the following three cases:

\noindent
{\bf Case 1.}
$a=0$. W.l.o.g we choose $\cB_v$ such that $z^1=e_1$, where $e_1$ is a unit
vector of length $n$. We have that $X_{v,z^1}=(-\hat{u},\textbf{0})$,
where $\textbf{0}$ is the all zeros vector of length ${n-1 \choose 2}$;
for all $z^i\in\cB_v, 2\le i\le n-1$ we have that
$X_{v,z^i}=(\textbf{0},X_{\hat{u},\hat{z}^i})$, where
$\hat{z}^i$ is the $(n-1)$-suffix of $\hat{z}^i$ and $\textbf{0}$
is the all zeros vector of length $n-1$. Hence,
Corollary~\ref{corollary:basisOfCodewords} implies
that $P_\Span{v} = \Span{\left\{X_{v,z^1},\ldots,X_{v,z^{n-1}}\right\}} = U_{\hat{u},0}$.

\noindent
{\bf Case 2.}
$a\ne 0, u\ne \textbf{0}$. For $i\ge 2$ we have:
\[
X_{v,e_i} = \left(\underbrace{0,\cdots,0,a,0,\cdots,0}_{\text{$n-1$ entries, $(i-1)$th equals $a$}},\hspace{0.7cm} \underbrace{\hspace{0.7cm} X_{\hat{u},\hat{e}_{i-1}} \hspace{0.7cm}} _{\text{${n-1 \choose 2}$ entries}} \right),
\]
where $\hat{e}_{i-1}$ is the $(n-1)$-suffix of the unit vector $e_i\in\bF_q^n$. Since $a\ne 0$ we may choose $\cB_v = \left\{e_2,\ldots,e_n\right\}$ and obtain by Corollary \ref{corollary:basisOfCodewords} that $P_\Span{v} = \Span{\left\{X_{v,e_2},\ldots,X_{v,e_n}\right\}} = U_{u,a^{-1}}$.

\noindent
{\bf Case 3.}
$a\ne 0,u=\textbf{0}$. For $i\ge 2$ we have:
\[
X_{v,e_i} = \left(\underbrace{0,\cdots,0,a,0,\cdots,0}_{\text{$n-1$ entries, $(i-1)$th equals $a$}},\textbf{0}\right).
\]
We may similarly choose $\cB_v = \left\{e_2,\ldots,e_n\right\}$ and obtain $P_\Span{v} = U_0$.
\end{proof}

Starting from an 1-intersecting code in $\cG_q(3,2)$
which consists of all the two-dimensional subspaces of
$\F_q^3$ we can obtain all the codes constructed in
Section~\ref{sec:construction} recursively. We note
that the initial condition consists exactly of all the
lines of the projective plane of order $q$.

\section{Conclusion and Problems for Future Research}
\label{sec:conclude}

We have made a discussion on the size of the largest
$t$-intersecting equidistant codes.
The largest codes are known to be the trivial sunflowers.
We discussed
trivial codes and surveyed the known results in this
direction. A construction
of non-sunflower 1-intersecting codes in $\cG_q(n,k)$,
$n \geq \binom{k+1}{2}$, whose size is $\sbinomq{k+1}{1}$,
based on the Pl\"{u}cker embedding, is given.
We showed that in at least one case there are larger
non-sunflower 1-intersecting equidistant codes.
Many important and usually very difficult problems
remained for future research. We list herein a few.

\begin{enumerate}
\item Find the size of the largest partial spread for
any given set of parameters.

\item Prove (or disprove) that the size
of a non-sunflower $t$-intersecting
constant dimension code of dimension $k$, where $t > 1$ and $k > t+1$,
is at most $\sbinomq{k+1}{1}$.

\item Identify the cases for which the size
of a non-sunflower $t$-intersecting
constant dimension code of dimension~$k$, where $t > 1$ and $k > t+1$,
is less than $\sbinomq{k+1}{1}$.

\item Identify the cases for which the size
of a non-sunflower 1-intersecting
constant dimension code of dimension $k$, where $k > 2$,
is greater than $\sbinomq{k+1}{1}$.

\item Find new constructions for non-sunflower $t$-intersecting
constant dimension codes of dimension $k$, where $t \geq 1$,
$k > t+1$, whose size is larger than the codes obtained from
partial spreads and their orthogonal codes.

\item Prove or disprove that the size of the largest equidistance
code with subspaces distance $d$ in~$\cP_q(n)$ depends on the size
of the largest equidistance code with subspace distance $d$ in
$\cG_q(n,k)$ for some~$k$.

\item Find 1-intersecting codes in $\cG_q(n,k)$ of size
$\sbinomq{k+1}{1}$, $k > 3$ and $n < \binom{k+1}{2}$.

\item Develop the theory of constant rank codes.

\item Find large equidistant rank metric codes.
\end{enumerate}


\end{document}